\documentclass[12pt]{amsart}

\usepackage{amsmath,amssymb,amsthm,color,ytableau}
\definecolor{shadecolor}{rgb}{.61,.87,1}

\DeclareMathOperator{\tr}{tr}
\DeclareMathOperator{\SL}{SL}
\DeclareMathOperator{\PSL}{PSL}

\DeclareMathOperator{\rank}{rank}

\newcommand{\C}{\mathbb{C}}

\newcommand{\E}{\mathbb{E}}
\newcommand{\R}{\mathbb{R}}

\newcommand{\F}{\mathbb{F}}

\newcommand{\ve}{\varepsilon}

\newcommand{\1}{^{-1}}

\newcommand{\f}[2]{\frac{#1}{#2}}

\newtheorem{theorem}{Theorem}[section]

\newtheorem{corollary}[theorem]{Corollary}
\newtheorem{proposition}[theorem]{Proposition}
\newtheorem{lemma}[theorem]{Lemma}

\numberwithin{theorem}{section}

\title{Mixing for three-term progressions in finite simple groups}
\author{Sarah Peluse}
\address{Department of Mathematics, Stanford University, Stanford, California 94305}
\email{speluse@stanford.edu}

\begin{document}

\begin{abstract}
Answering a question of Gowers, Tao proved that any $A\times B\times C\subset \SL_d(\F_q)^3$ contains $\f{|A||B||C|}{|\SL_d(\F_q)|}+O_d(\f{|\SL_d(\F_q)|^2}{q^{\min(d-1,2)/8}})$ three-term progressions $(x,xy,xy^2)$. Using a modification of Tao's argument, we prove such a mixing result for three-term progressions in all nonabelian finite simple groups except for $\PSL_2(\F_q)$ with an error term that depends on the degree of quasirandomness of the group. This argument also gives an alternative proof of Tao's result when $d>2$, but with the error term $O(\f{|\SL_d(\F_q)|^2}{q^{(d-1)/24}})$.
\end{abstract}

\maketitle

\section{Introduction}

Let $G$ be a finite group with no nontrivial representations of degree less than $D$. Such groups are called \textit{$D$-quasirandom}. In~\cite{G}, Gowers showed that if $G$ is $D$-quasirandom, then
\[
\#\{(x,y,xy)\in A\times B\times C: x,y\in G\}=\f{|A||B||C|}{|G|}+O\left(\sqrt{\f{|G||A||B||C|}{D}}\right)
\]
for every $A,B,C\subset G$. The quantity $|A||B||C|/|G|$ is the number of triples $(x,y,xy)$ that we would expect to lie in the Cartesian product of three random subsets of $G$ of the same size as $A$, $B$, and $C$. Thus, whenever $|A||B||C|$ is sufficiently large, the set $A\times B\times C$ contains close to the expected number of triples $(x,y,xy)$.

Gowers asked in~\cite{G} whether a similar result holds in quasirandom groups for configurations other than $(x,y,xy)$, in particular for three-term progressions $(x,xy,xy^2)$. This was answered in the affirmative by Tao in~\cite{T} for the family of groups $\SL_d(\F_q)$ with $d$ bounded, all of whose nontrivial representations have degree $\gg q^{d-1}$ (see~\cite{LaS}). Let $r_{G}(A,B,C)$ denote the number of three-term progressions in $A\times B\times C$ for $A,B,C\subset G$. Tao showed that
\begin{equation}\label{t}
r_{\SL_d(\F_q)}(A,B,C)=\f{|A||B||C|}{|\SL_d(\F_q)|}+\begin{cases}O\left(\f{|\SL_d(\F_q)|^2}{q^{1/8}}\right) & d=2 \\ O_d\left(\f{|\SL_d(\F_q)|^2}{q^{1/4}}\right) & d>2 \end{cases}
\end{equation}
for any $A,B,C\subset\SL_d(\F_q)$. His proof relies on algebraic geometry, and thus does not seem to generalize to prove a similar result for the family of alternating groups $A_n$, for example, which are $(n-1)$-quasirandom.

In this note, we give a simple modification of Tao's argument that produces a nontrivial bound for $|r_G(A,B,C)-|A||B||C|/|G||$ when $G$ is any nonabelian finite simple group other than $\PSL_2(\F_q)$. This argument also recovers~(\ref{t}) when $d>2$, but with the error term $O(|\SL_d(\F_q)|^2/q^{(d-1)/24})$.

Letting $\hat{G}$ denote the set of irreducible unitary representations of $G$, $1$ the trivial representation of $G$, and $d_\rho$ the degree of $\rho\in\hat{G}$, we have the following theorem:
\begin{theorem}\label{main}
Let $G$ be a finite group and $A,B,C\subset G$. Then
\begin{equation}\label{bd}
r_{G}(A,B,C)=\f{|A||B||C|}{|G|}+O\left(\left(\sum_{1\neq\rho\in\hat{G}}\f{1}{d_\rho}\right)^{1/8}|G|^2\right).
\end{equation}
\end{theorem}
The error term in~(\ref{bd}) is too large to be useful when $G=\SL_2(\F_q)$ or $\PSL_2(\F_q)$, since in these cases $G$ has $\asymp q$ representations of degree $\asymp q$. However, by work of Liebeck and Shalev~\cite{LS1}~\cite{LS3}, the quantity $\sum_{1\neq\rho\in\hat{G}}1/d_\rho$ is quite small for many other groups of interest. Combining Theorem~\ref{main} with Liebeck and Shalev's bounds and with Tao's result for $\PSL_2(\F_q)$ yields the following corollary:
\begin{corollary}\label{cor1}
Let $G$ be a finite simple group that is $D$-quasirandom. Then
\[
r_G(A,B,C)=\f{|A||B||C|}{|G|}+O\left(\f{|G|^2}{D^{1/24}}\right)
\]
for every $A,B,C\subset G$.
\end{corollary}
One can use the bounds from~\cite{LS1} and~\cite{LS3} to get a better error term than the one in Corollary~\ref{cor1} for some specific groups.
\begin{corollary}\label{cor2}
We have that
\begin{equation}\label{an}
r_{A_n}(A,B,C)=\f{|A||B||C|}{|A_n|}+O\left(\f{|A_n|^2}{n^{1/8}}\right)
\end{equation}
for all $A,B,C\subset A_n$, that
\begin{equation}\label{sld}
r_{\SL_d(\F_q)}(A,B,C)=\f{|A||B||C|}{|\SL_d(\F_q)|}+O_d\left(\f{|\SL_d(\F_q)|^2}{q^{(d-1)(1-2/d)/8}}\right),
\end{equation}
for all $A,B,C\subset\SL_d(\F_q)$, and that, for any $\ve>0$, there exists a $d(\ve)>0$ such that whenever $d\geq d(\ve)$, we have
\begin{equation}\label{sl}
r_{\SL_d(\F_q)}(A,B,C)=\f{|A||B||C|}{|\SL_d(\F_q)|}+O\left(\f{|\SL_d(\F_q)|^2}{q^{(d-1)(1-\ve)/8}}\right)
\end{equation}
for all $A,B,C\subset \SL_d(\F_q)$.
\end{corollary}
It appears that the most difficult case of bounding $|r_{\SL_d(\F_q)}(A,B,C)-|A||B||C|/|\SL_d(\F_q)||$ is when $d=2$. Here Theorem~\ref{main} is useless, and Tao needed to use properties of $\SL_2(\F_q)$ other than the degrees of its irreducible representations to get a nontrivial bound in~\cite{T}.

This paper is organized as follows. In Section~\ref{not}, we set notation and review some basic facts about quasirandom groups. We prove Theorem~\ref{main} in Section~\ref{thm}, and then show how to deduce Corollary~\ref{cor1} and Corollary~\ref{cor2} from the bounds of Liebeck and Shalev in Section~\ref{cors}.

\section*{Acknowledgments}
The author thanks Kannan Soundararajan for many helpful comments on earlier drafts of this note and Persi Diaconis for help with references. The author is supported by the National Science Foundation Graduate Research Fellowship Program under Grant No. DGE-114747 and by the Stanford University Mayfield Graduate Fellowship.

\section{Notation and preliminaries}\label{not}
Let $G$ be a finite group. The indicator function and normalized indicator function of a set $S\subset G$ will be denoted by $1_S$ and $\mu_S=\f{1}{|S|}1_S$, respectively. For all functions $f_1,f_2: G\to\R$ and $S\subset G$, we write
\[
\|f_1\|_p^p=\sum_{x\in G}|f_1(x)|^p,
\]
\[
\E_{x\in S}f_1(x)=\f{1}{|S|}\sum_{x\in S}f_1(x),
\]
and
\[
(f_1*f_2)(x)=\sum_{y\in G}f_1(xy\1)f_2(y).
\]
We will also write $\E_{x}$ instead of $\E_{x\in G}$ when averaging over all of $G$.

The Fourier transform of $f_1$ at $\rho\in\hat{G}$ is given by
\[
\hat{f_1}(\rho)=\sum_{x\in G}f_1(x)\rho(x).
\]
As usual, we have that $\widehat{f_1*f_2}(\rho)=\hat{f_1}(\rho)\hat{f_2}(\rho)$ for every $\rho\in\hat{G}$, and also Parseval's identity:
\[
\|f_1\|_2^2=\sum_{x\in G}f_1(x)^2=\f{1}{|G|}\sum_{\rho\in\hat{G}}d_\rho\|\hat{f_1}(\rho)\|_{HS}^2,
\]
where $\|\cdot\|_{HS}$ is the Hilbert-Schmidt norm $\|M\|_{HS}^2=\tr(M^*M)$, which is submultiplicative. See~\cite{D}, for example, for background on nonabelian Fourier analysis.

Inequality~(\ref{two}) in the following lemma is due to Babai, Nikolov, and Pyber~\cite{BNP}. Inequality~(\ref{one}) is just a tiny modification of it that will be necessary in the proof of Theorem~\ref{main}.
\begin{lemma}\label{convolution}
Suppose that $f_1,f_2:G\to\R$ and $\E_{x}f_1(x)=0$. Then
\begin{equation}\label{one}
\|f_1*f_2\|_2\leq \|f_1\|_2\left(\sum_{1\neq\rho\in\hat{G}}\|\hat{f_2}(\rho)\|_{HS}^2\right)^{1/2}.
\end{equation}
If, in addition, $G$ is $D$-quasirandom, then
\begin{equation}\label{two}
\|f_1*f_2\|_2\leq\f{|G|^{1/2}}{D^{1/2}}\|f_1\|_2\|f_2\|_2.
\end{equation}
\end{lemma}
\begin{proof}
By Parseval's identity, for any $\rho\in\hat{G}$ we have
\[
\|f_1\|_2^2=\f{1}{|G|}\sum_{\xi\in\hat{G}}d_\xi\|\hat{f_1}(\xi)\|_{HS}^2\geq\f{d_\rho}{|G|}\|\hat{f_1}(\rho)\|_{HS}^2,
\]
so that $\|\hat{f_1}(\rho)\|_{HS}^2\leq\f{|G|}{d_\rho}\|f_1\|_2^2$. Since $\hat{f_1}(1)=0$, Parseval's identity again implies that
\begin{align*}
\|f_1*f_2\|_2^2 &= \f{1}{|G|}\sum_{1\neq\rho\in\hat{G}}d_\rho\|\hat{f_1}(\rho)\hat{f_2}(\rho)\|_{HS}^2 \\
&\leq \f{1}{|G|}\sum_{1\neq\rho\in\hat{G}}d_\rho\|\hat{f_1}(\rho)\|_{HS}^2\|\hat{f_2}(\rho)\|_{HS}^2 \\
&\leq\|f_1\|_2^2\sum_{1\neq\rho\in\hat{G}}\|\hat{f_2}(\rho)\|_{HS}^2,
\end{align*}
giving~(\ref{one}). Writing
\[
\sum_{1\neq\rho\in\hat{G}}\|\hat{f_2}(\rho)\|_{HS}^2=\f{|G|}{D} \f{1}{|G|}\sum_{1\neq\rho\in\hat{G}}D\|\hat{f_2}(\rho)\|_{HS}^2\leq\f{|G|}{D} \f{1}{|G|}\sum_{1\neq\rho\in\hat{G}}d_\rho\|\hat{f_2}(\rho)\|_{HS}^2
\]
when $G$ is $D$-quasirandom shows that~(\ref{two}) follows from~(\ref{one}).
\end{proof}
We will also need the following easy corollary of inequality~(\ref{two}) of Lemma~\ref{convolution}, which is Lemma 1.3 of~\cite{T}:
\begin{corollary}\label{avg}
Let $G$ be a $D$-quasirandom group and suppose that $f:G\to\R$ has mean zero. Then
\[
\E_{x\in G}|\E_{y\in G}f(y)f(yx)|\leq \f{\|f\|_2^2}{D^{1/2}|G|}.
\]
\end{corollary}

\section{Proof of Theorem~\ref{main}}\label{thm}
Following Tao~\cite{T}, we define the following trilinear form on the space of real-valued functions on $G$:
\[
\Lambda(f_1,f_2,f_3)=\E_{x,y\in G}f_1(x)f_2(xy)f_3(xy^2).
\]
Note that for any $A,B,C\subset G$, we have
\[
\Lambda(1_A,1_B,1_C)=\f{\#\{(x,xy,xy^2)\in A\times B\times C:x,y\in G\}}{|G|^2},
\]
so that $\Lambda(1_A,1_B,1_C)=|G|^{-2}r_{G}(A,B,C)$ is the normalized count of three-term progressions in $A\times B\times C$. Setting $f_C=1_C-\f{|C|}{|G|}$, we get
\[
|\Lambda(1_A,1_B,f_C)|=\left|\Lambda(1_A,1_B,1_C)-\f{|A||B||C|}{|G|^3}\right|,
\]
so that in order to bound the difference between the actual number of three-term progressions in $A\times B\times C$ and the expected number, it suffices to bound $|\Lambda(f_1,f_2,f_3)|$ when $f_3$ has mean zero. We do this in Proposition~\ref{prop} below, from which Theorem~\ref{main} follows immediately.

When $G$ is abelian, multilinear forms like $\Lambda$ can be bounded above in terms of Gowers norms of the arguments by repeatedly changing variables and applying Cauchy-Schwarz. Tao's idea in~\cite{T} was to try to carry out this strategy for $\Lambda$ when $G$ is not abelian. The first part of the proof of Proposition~\ref{prop} is essentially identical to the proof of Tao's Proposition 2.2 in~\cite{T}. Tao actually bounds $\E_{y}|\E_{x}f_1(x)f_2(xy)f_3(xy^2)|$, which yields a mixing result for the configuration $(x,xy,xy^2,y)$. Bounding $\E_{y}|\E_{x}f_1(x)f_2(xy)f_3(xy^2)|$, however, does not lead to an expression that is nice from the point of view of representation theory.

The main difference between our argument and Tao's is as follows. In the course of the proof, measures of the form $\mu_{gC(g)}$ appear, where $C(g)$ denotes the conjugacy class of $g\in G$. In Tao's argument, these measures are averaged over cosets of centralizers of elements in $G$ in order to get a measure that is closer to uniform. Instead, we use that $\mu_{gC(g)}$ is a translate of a class function, and thus has a very simple Fourier transform. Indeed, $\hat{\mu}_{gC(g)}(\rho)=\rho(g)\hat{\mu}_{C(g)}(\rho)$ for all $g\in G$ and $\rho\in\hat{G}$. Since $\mu_{C(g)}$ is a class function, Schur's lemma implies that $\hat{\mu}_{C(g)}(\rho)=c_{\rho,g}I_{d_\rho}$ for some $c_{\rho,g}\in\C$ depending only on $\rho$ and the conjugacy class of $g$. Taking the trace of both sides of
\[
\f{1}{|C(g)|}\sum_{h\in C(g)}\rho(h)=c_{\rho,g}I_{d_\rho}
\]
reveals that $c_{\rho,g}=\f{\chi_{\rho}(g)}{d_\rho}$, where $\chi_\rho$ denotes the character of $\rho$. Thus, we have $\hat{\mu}_{gC(g)}(\rho)=\f{\chi_{\rho}(g)}{d_\rho}\rho(g)$.

Another difference between the proofs is that, instead of using inequality~(\ref{two}) from Lemma~\ref{convolution} like Tao, we use inequality~(\ref{one}). In all nonabelian finite simple groups other than $\PSL_2(\F_q)$, most of the irreducible representations have degree much larger than the degree of quasirandomness. The use of~(\ref{one}) to take advantage of this is the reason why we can get error terms in Corollaries~\ref{cor1} and~\ref{cor2} that decay polynomially in the degree of quasirandomness of the group.
\begin{proposition}\label{prop}
Let $G$ be a finite group and suppose that $f_1,f_2,f_3: G\to\R$ with $\|f_1\|_\infty=\|f_2\|_\infty=\|f_3\|_\infty=1$ and $\E_{x}f_3(x)=0$. Then
\[
|\Lambda(f_1,f_2,f_3)|\ll\left(\sum_{1\neq\rho\in\hat{G}}\f{1}{d_\rho}\right)^{1/8}.
\]
\end{proposition}
\begin{proof}
Making the change of variables $x\mapsto xy\1$, we have
\[
\Lambda(f_1,f_2,f_3)=\E_{x\in G}f_2(x)\E_{y\in G}f_1(xy\1) f_3(xy).
\]
Applying Cauchy-Schwarz in the $x$ variable thus implies that
\[
|\Lambda(f_1,f_2,f_3)|^2\leq\E_{x,y_1,y_2\in G}f_1(xy_1\1)f_1(xy_2\1) f_3(xy_1)f_3(xy_2).
\]
After the change of variables $xy_1\1\mapsto x$, $y_1\mapsto y$, and $y_1y_2\1\mapsto a$, the right-hand side above becomes
\[
\E_{x,y,a}f_1(x)f_1(xa)f_3(xy^2)f_3(xya\1 y)=\E_{x,a}\Delta_af_1(x)\E_{y}\Delta_{y\1 a\1 y}f_3(xy^2),
\]
where $\Delta_cf(z)=f(z)f(zc)$.

Applying Cauchy-Schwarz in $x$ and $a$ yields
\[
|\Lambda(f_1,f_2,f_3)|^4\leq \E_{x,a,y_1,y_2}\Delta_{y_1\1 a\1 y_1}f_3(xy_1^2)\Delta_{y_2\1 a\1 y_2}f_3(xy_2^2).
\]
After the change of variables $xy_1^2\mapsto x, y_1\1 a\1 y_1\mapsto b,y_1\mapsto y,$ and $y_1\1 y_2\mapsto g$, the right-hand side above is
\[
\E_{x,y,b,g}\Delta_bf_3(x)\Delta_{g\1 bg}f_3(xy\1 g y g).
\]

Now,
\begin{align*}
\E_{y}\Delta_{g\1 b g}f_3(xy\1 g y g) &=\f{1}{|G|}\sum_{z\in G}\Delta_{g\1 b g}f_3(xz\1)1_{\{g\1 y\1 g\1 y:y\in G\}}(z) \\
&=\sum_{z\in G}\Delta_{g\1 b g}f_3(xz\1)\f{|Z(g\1)|}{|G|}1_{g\1 C(g\1)}(z),
\end{align*}
where $Z(g)$ denotes the centralizer of $g$ in $G$. Thus, since $|G|/|Z(g\1)|=|C(g\1)|$, we have that $\E_{y}\Delta_{g\1 b g}f_3(xy\1 g y g)=(\Delta_{g\1 b g}f_3*\mu_{g\1 C(g\1)})(x)$. Hence, 
\[
\E_{x,y,b,g}\Delta_{b}f_3(x)\Delta_{g\1 b g}f_3(xy\1 g y g)=\E_{x,b,g}\Delta_bf_3(x)(\Delta_{g\1 bg}f_3*\mu_{g\1 C(g\1)})(x).
\]

Now we diverge from Tao's argument and do not average $\mu_{g\1 C(g\1)}$ over $g$ in some coset of $Z(b)$. Writing $\Delta_{g\1 b g}f_3$ as $f_{g\1 b g}+\delta_{g\1 b g}$ where $\delta_{g\1 b g}=\E_{z}\Delta_{g\1 b g}f_3(z)$, we have
\[
\E_{b,g}|\E_x\Delta_{b}f_3(x)(\Delta_{g\1 b g}f_3*\mu_{g\1 C(g\1)})(x)| \leq \E_{b,g}|\E_x\Delta_{b}f_3(x)(f_{g\1 b g}*\mu_{g\1 C(g\1)})(x)|+D^{-1/2}
\]
if $G$ is $D$-quasirandom. This is because $f_3$ has mean zero and $\|\Delta_{g\1 b g}f_3\|_\infty\leq 1$, so that
\[
\E_{b,g}|\E_x\Delta_{b}f_3(x)(\delta_{g\1 b g}*\mu_{g\1 C(g\1)})(x)|\leq D^{-1/2}
\]
by Corollary~\ref{avg}.

By Cauchy-Schwarz, we have
\[
|\E_x\Delta_{b}f_3(x)(f_{g\1 b g}*\mu_{g\1 C(g\1)})(x)|\leq \f{1}{|G|^{1/2}}\|f_{g\1 b g}*\mu_{g\1 C(g\1)}\|_2,
\]
since $\|\Delta_b f_3\|_\infty\leq 1$. Inequality~(\ref{one}) from Lemma~\ref{convolution} says that
\[
\|f_{g\1 b g}*\mu_{g\1 C(g\1)}\|_2\leq |G|^{1/2}\left(\sum_{1\neq\rho\in\hat{G}}\|\hat{\mu}_{g\1 C(g\1)}(\rho)\|_{HS}^2\right)^{1/2},
\]
since $\|f_{g\1 b g}\|_2\leq \|\Delta_{g\1 b g}f_3\|_2\leq |G|^{1/2}$ and $f_{g\1 b g}$ has mean zero. Thus,
\[
\E_{b,g}|\E_x\Delta_{b}f_3(x)(f_{g\1 b g}*\mu_{g\1 C(g\1)})(x)|\leq \E_{g}\left(\sum_{1\neq\rho\in\hat{G}}\|\hat{\mu}_{g\1 C(g\1)}(\rho)\|_{HS}^2\right)^{1/2}.
\]
Another application of Cauchy-Schwarz in the $g$ variable yields
\[
\E_{b,g}|\E_x\Delta_{b}f_3(x)(f_{g\1 b g}*\mu_{g\1 C(g\1)})(x)|\leq \left(\E_{g}\sum_{1\neq\rho\in\hat{G}}\|\hat{\mu}_{gC(g)}(\rho)\|_{HS}^2\right)^{1/2},
\]
after making the change of variables $g\mapsto g\1$.

Now, $\hat{\mu}_{gC(g)}(\rho)=\f{\chi_\rho(g)}{d_\rho}\rho(g)$ and $\rho(g)^*\rho(g)=I_{d_\rho}$, so that 
\[
\|\hat{\mu}_{gC(g)}(\rho)\|_{HS}^2=\left|\f{\chi_{\rho}(g)}{d_\rho}\right|^2\tr(\rho(g)^*\rho(g))=\left|\f{\chi_{\rho}(g)}{d_\rho}\right|^2d_\rho.
\]
Thus,
\[
\E_{g}\sum_{1\neq\rho\in\hat{G}}\|\hat{\mu}_{gC(g)}(\rho)\|_{HS}^2=\sum_{1\neq\rho\in\hat{G}}\f{1}{d_\rho}\E_{g}\chi_{\rho}(g)\overline{\chi_{\rho}(g)}=\sum_{1\neq\rho\in\hat{G}}\f{1}{d_\rho}.
\]
We conclude that
\[
|\Lambda(f_1,f_2,f_3)|^4\leq \left(\sum_{1\neq\rho\in\hat{G}}\f{1}{d_\rho}\right)^{1/2}+\f{1}{D^{1/2}}\leq 2\left(\sum_{1\neq\rho\in\hat{G}}\f{1}{d_\rho}\right)^{1/2},
\]
because $D^{-1}\leq \sum_{1\neq\rho\in\hat{G}}1/d_\rho$.
\end{proof}

\section{Deduction of Corollaries~\ref{cor1} and~\ref{cor2}}\label{cors}
The \textit{Witten zeta function} $\zeta_G$ of a group $G$ is defined by
\[
\zeta_G(s) = \sum_{\rho\in\hat{G}}\f{1}{d_\rho^s}.
\]
Note that if $G$ is $D$-quasirandom, then $\zeta_G(1)-1\leq D^{-\delta}(\zeta_G(1-\delta)-1)$ for every $0<\delta\leq 1$.

There is an extensive literature on bounds for $\zeta_G(s)-1$ for various finite groups $G$ and applications of such bounds to problems in group theory; see~\cite{S} for a survey. We will use Liebeck and Shalev's bounds from~\cite{LS1} and~\cite{LS3}. Lemma 2.7 of~\cite{LS1} says that
\[
\zeta_{A_n}(s)=1+O(n^{-s})
\]
for all $s>0$, which, combined with Theorem~\ref{main}, immediately implies~(\ref{an}).

The relevant bounds for $\zeta_G(s)-1$ when $G$ is a group of Lie type can be found in~\cite{LS3}. The \textit{Coxeter number} of a simple algebraic group $G$ is the quantity $\f{\dim G}{\rank G}-1$. Liebeck and Shalev's first theorem in~\cite{LS3} concerns groups of bounded rank:
\begin{theorem}[Liebeck and Shalev, Theorem 1.1 of~\cite{LS3}]\label{ls1}
Fix a Lie type $L$ and let $h$ be the Coxeter number of the corresponding simple algebraic group. Let $G=G(q)$ be a finite quasisimple group of type $L$ over $\F_q$. Then for any fixed $s>2/h$, we have that
\[
\zeta_G(s)\to 1
\] 
as $q\to\infty$. In addition, we have that $1\ll_{L}\zeta_G(2/h)\ll_{L} 1$.
\end{theorem}
The Coxeter number corresponding to $\SL_d(\F_q)$ is $\f{d^2-1}{d-1}-1=d$. Thus, Theorem~\ref{ls1} tells us that
\[
\zeta_{\SL_d(\F_q)}(1)-1\ll q^{-(d-1)(1-2/d)}(\zeta_{\SL_d(\F_q)}(2/d)-1)\ll_{d}q^{-(d-1)(1-2/d)},
\]
since $\SL_d(\F_q)$ is $\gg q^{d-1}$-quasirandom. This bound and Theorem~\ref{main} together show~(\ref{sld}).

Liebeck and Shalev's second theorem concerns collections of groups with unbounded rank:
\begin{theorem}[Liebeck and Shalev, Theorem 1.2 of~\cite{LS3}]\label{ls2}
Fix $s>0$. There exists an integer $r(s)$ such that for any finite quasisimple group of Lie type $G$ of rank at least $r(s)$, we have that
\[
\zeta_G(s)\to 1
\]
as $|G|\to\infty$.
\end{theorem}
The group $\SL_d(\F_q)$ has rank $d-1$, so that for every $\ve>0$, Theorem~\ref{ls2} furnishes a $d(\ve)$ such that whenever $d\geq d(\ve)$, we have that $\zeta_{\SL_d(\F_q)}(\ve)-1\leq 1$. Then $\zeta_{\SL_d(\F_q)}(1)-1\ll q^{-(1-\ve)(d-1)}$, since $\SL_d(\F_q)$ is $\gg q^{d-1}$-quasirandom. This bound and Theorem~\ref{main} imply~(\ref{sl}), completing the proof of Corollary~\ref{cor2}.

To prove Corollary~\ref{cor1}, we use the classification of finite simple groups. Let $r(2/3)$ be as in Theorem~\ref{ls2}. Then for all finite simple groups of Lie type of rank at least $r(2/3)$, we have that $\zeta_G(2/3)-1\ll 1$.

 To deal with the finitely many families of finite simple groups of Lie type of rank less than $r(2/3)$, we apply Theorem~\ref{ls1}. The only family of finite simple groups of Lie type whose corresponding Coxeter number is less than $3$ is $\PSL_2(\F_q)$ (see~\cite{C}.) Thus, Theorem~\ref{ls1} implies that if $G$ is a finite simple group of Lie type of rank less than $r(2/3)$, then $\zeta_G(2/3)-1\ll 1$ as well, since such groups $G$ can have one of only finitely many Lie types.

Since $\zeta_{A_n}(2/3)-1\ll n^{-2/3}\ll 1$, we thus have that
\[
\zeta_{G}(2/3)-1\ll 1
\]
for all nonabelian finite simple groups $G$ other than $\PSL_2(\F_q)$. It follows that if $G\neq\PSL_2(\F_q)$ is a $D$-quasirandom nonabelian finite simple group, then $\zeta_G(1)-1\ll D^{-1/3}$. This combined with Theorem~\ref{main} proves Corollary~\ref{cor1} when $G\neq\PSL_2(\F_q)$, and the case $G=\PSL_2(\F_q)$ follows from Tao's result~(\ref{t}). Of course, Corollary~\ref{cor1} is trivial for abelian groups.
\bibliographystyle{plain}
\bibliography{bib}

\end{document}